\RequirePackage[l2tabu, orthodox]{nag}

\documentclass[12pt]{amsart}
\usepackage{fullpage,url,amssymb,enumerate}
\usepackage[all]{xy} 
\usepackage{comment}
\usepackage{graphicx}

\usepackage[OT2,T1]{fontenc}

\usepackage{amsthm}

\usepackage{color}

\def\bC{\mathbb{C}}

\def\cH{\mathcal{H}}

\def\cM{\mathcal{M}}

\def\cO{\mathcal{O}}
\def\bP{\mathbb{P}}
\def\cP{\mathcal{P}}

\def\bZ{\mathbb{Z}}

\def\wt{\widetilde}

\DeclareMathOperator{\rat}{rat}

\DeclareMathOperator{\Spec}{Spec}

\newtheorem{thm}{Theorem}[section]
\newtheorem{lem}[thm]{Lemma}
\newtheorem{cor}[thm]{Corollary}
\newtheorem{prop}[thm]{Proposition}

\newtheorem{defn}[thm]{Definition}

\makeatletter
\g@addto@macro\bfseries{\boldmath} 
\makeatother

\usepackage{microtype}

\usepackage[
	backref,
	pdfauthor={Carl Lian, Riccardo Moschetti}, 
	pdftitle={Z/rZ-equivariant covers of P1 with moving ramification},
]{hyperref}

\begin{document}

\title{$\bZ/r\bZ$-equivariant covers of $\bP^1$ with moving ramification}

\author{Carl Lian}
\address{Institut f\"{u}r Mathematik, Humboldt-Universit\"{a}t zu Berlin, 12489 Berlin, Germany}
\email{liancarl@hu-berlin.de}

\author{Riccardo Moschetti}
\address{ Dipartimento di Matematica ``F. Casorati'', Universit\`{a} degli Studi di Pavia, Via Ferrata 5, 27100 Pavia, Italy}
\email{riccardo.moschetti@unipv.it}

\date{\today}

\begin{abstract}
    Let $X\to\bP^1$ be a general cyclic cover. We give a simple formula for the number of equivariant meromorphic functions on $X$ subject to ramification conditions at variable points. This generalizes and gives a new proof of a recent result of the second author and Pirola on \textit{hyperelliptic odd covers}.
\end{abstract}

\maketitle

\section{Introduction}

Let $\cH$ be a moduli space of branched covers of $\bP^1$, and let $\overline{\cH}$ be its compactification by Harris-Mumford admissible covers, see \cite{hm}. Let $\phi:\overline{\cH}\to\overline{\cM}_{g,n}$ be the map remembering the source of a branched cover, possibly with additional marked points (e.g. ramification points). 

Pushing forward the fundamental class by $\phi$ yields a cycle class on $\overline{\cM}_{g,n}$, which, by a theorem of Faber-Pandharipande, lies in the tautological ring, see \cite{fp}. While an algorithm is given to compute this class, it is in general quite impractical to implement, and few explicit formulas are known.

In the case where $\dim(\cH)=\dim(\cM_{g,n})$, so that $\phi$ is generically finite, computing this cycle class amounts to computing the degree of $\phi$, for which formulas were given (assuming the ramification points lie in distinct fibers) in \cite{lian_pencils} using limit linear series and Schubert calculus techniques. By definition, these degrees count maps $h:X\to\bP^1$ out of a fixed curve with ramification conditions imposed at possibly moving (variable) points.

A variant of this problem has been studied in \cite{mp}, in which $X$ is assumed to be hyperelliptic and $h$ is required to be equivariant for the hyperelliptic action. More precisely, it is proven that:

\begin{thm}\cite{mp}\label{hyperelliptic_thm}
Let $X$ be a general hyperelliptic curve of genus $g$ with Weierstraß points $P_1,\ldots,P_{2g+2}$. Fix integers $n_1,\ldots,n_{2g+2}\in\bZ_{\ge0}$ such that $\sum_{i=1}^{2g+2}n_i=g-1$. Consider all covers $h:X\to\bP^1$ of degree $4g$ such that:
\begin{itemize}
    \item $h$ is equivariant for the hyperelliptic involution on $X$ and the involution $z\mapsto -z$ on $\bP^1$,
    \item $h^{-1}(\infty)=\sum_{i=1}^{2g+2}(2n_i+1)P_i$,
    \item $h$ is unramified over $0\in\bP^1$, and 
    \item all remaining ramification points of $h$ are triple points.
\end{itemize}
Then, for any choice of $n_i$, the number of such covers (called \emph{hyperelliptic odd covers}) is $2^{2g}$.
\end{thm}

Here, we have ramification conditions imposed at the \textit{fixed} points $P_i$, in addition to \textit{moving} (unspecified) triple ramification points away from $0,\infty$. The enumerative count may also be expressed as the degree of the morphism $\phi_{\text{HOC}}:\cH^{\text{HOC}}_{g,n_i}\to\cH_g$, where $\cH^{\text{HOC}}_{g,n_i}$ is the moduli space of hyperelliptic odd covers for the given choice of integers $n_i$, $\cH_g$ is the stack of hyperelliptic curves, and $\phi_{\text{HOC}}$ remembers the source curve. The proof expresses such covers as solutions to a certain differential equation and studies these solutions in terms of periods on $X$. 

In this work, we employ a more straightforward technique to give a natural generalization of this problem to arbitrary cyclic covers. Fix a general cyclic cover $\sigma:X\to\bP^1$ of degree $r$ with marked ramification points $P_{i,j}\in X$; the $\bZ/r\bZ$-orbits of the marked points are the sets $\{P_{i,j}\}$ where $i$ is fixed and $j$ varies. We will denote by $\cH_{g,\bZ/r\bZ,\xi}$ the moduli space parametrizing such $\sigma$, where $\xi$ is a collection of elements in $\bZ/r\bZ$ describing the monodromy over the branch points of $\sigma$, see \S\ref{cyclic_covers}. 

We enumerate $\bZ/r\bZ$-equivariant maps $h:X\to\bP^1$, where $\bZ/r\bZ$ acts on the target by multiplication of a coordinate $z$ by $r$-th roots of unity. We require $h$ to have two types of ramification points:
\begin{itemize}
    \item ramification at $P_{i,j}$ over $0,\infty\in\bP^1$, with common ramification indices (for fixed $i$) given by integers $|a_i|>0$; we take $a_i>0$ whenever $h(P_{i,j})=0$ and $a_i<0$ otherwise.
    \item ramification away from $0,\infty\in\bP^1$ at unspecified (and unordered) points of $X$ occurring in $k$ orbits of size $r$, with ramification indices given by the multiset $B=\{b_1,\ldots,b_k\}$.
\end{itemize}
We will denote by $\cH_{g,\bZ/r\bZ,\xi,B}$ the moduli space parametrizing such $h$, see \S\ref{Zr_equivariant_covers}. 

Counting $\bZ/r\bZ$-equivariant covers amounts to computing the degrees of the forgetful maps $\phi_B:\cH_{g,\bZ/r\bZ,\xi,B}\to \cH_{g,\bZ/r\bZ,\xi}$. Under appropriate conditions detailed in \S\ref{setup}, these maps are generically quasi-finite and the spaces are non-empty. Our main theorem is as follows (see also Theorem \ref{main_thm}):

\begin{thm}\label{main_thm_intro}
Under appropriate numerical constraints ensuring that $\cH_{g,\bZ/r\bZ,\xi,B}$ and $\cH_{g,\bZ/r\bZ,\xi}$ are non-empty of the same dimension (see (\ref{eq:RH_cyclic}), (\ref{eq:specialfibers_deg}), (\ref{eq:k_vs_m}), (\ref{eq:RH_h}), (\ref{eq:t_vs_b}), and Corollary \ref{cor:cong}), the number of $\bZ/r\bZ$-equivariant covers $h:X\to\bP^1$ of degree $d$ with ramification indices $|a_i|$ over $0,
\infty$ is:
\begin{equation*}
\binom{(t_0+t_\infty)/r}{t_0/r}\cdot \prod_{j=1}^{k}(b_j-1)\cdot\frac{k!}{\prod_{\ell\in L'}c'_\ell!}
\end{equation*}
where:
\begin{itemize}
    \item $t_0,t_\infty$ are the numbers of points in $h^{-1}(0),h^{-1}(\infty)$, respectively, that are not among the marked ramification points $P_{i,j}$ of $\sigma$, and
    \item $L'$ is the set of distinct ramification indices $b_j$ occurring away from $0,\infty\in\bP^1$, and for $\ell\in L'$, $c'_{\ell}$ is the number of orbits of ramification points with ramification index $\ell$.\footnote{We will see later that it is somewhat more natural to consider the set $L$ of distinct integers among $b_j-1$.}
\end{itemize}
\end{thm}

Among the required numerical constraints is that the (signed) ramification indices $a_i$ of $h$ above $0,\infty$ satisfy a certain congruence condition in terms of the monodromy of $\sigma$, see Corollary \ref{cor:cong}. As long as this condition is satisfied, the final formula does not depend on the $a_i$; this is also the case in Theorem \ref{hyperelliptic_thm}.

Theorem \ref{hyperelliptic_thm} is the special case of Theorem \ref{main_thm_intro} where $r=2$, $t_0=4g$, $t_\infty=0$, $k=2g$, and the $b_i$ are all equal to 3. In particular, this covers also the case of elliptic curves with an odd spin structure in degree four, which has been addressed in \cite{fmnp} and \cite{lian_pencils}. 

For the proof of Theorem \ref{main_thm_intro}, the key observation is that the equivariant covers in question may be enumerated after passing to the scheme-theoretic quotients by the $\bZ/r\bZ$-actions on the source and target, so we reduce to a question of counting maps $h':\bP^1\to\bP^1$ with certain ramification profiles. We address this question in \S\ref{counting_meromorphic_functions}: the moving ramification conditions are imposed by way of Segre and Veronese embeddings, and we conclude our main result by analyzing the intersection that arises.

Lastly, we remark that we expect the method to work for equivariant covers for more general finite subgroups of $\text{PGL}_{2}(\bC)$, though we have not pursued this here.

\subsection{Acknowledgments}

C.L. was supported by an NSF Postdoctoral Fellowship, grant DMS-2001976. R.M. was supported by MIUR: Dipartimenti di Eccellenza Program (2018-2022)-Dept. of Math. Univ. of Pavia and by PRIN Project \emph{Moduli spaces and Lie theory} (2017). We thank Gavril Farkas, Pietro Pirola and Johannes Schmitt for helpful comments and correspondence.

\section{Setup}\label{setup}

In this section, we introduce the various moduli spaces and the relations between them. We work over $\bC$.

\subsection{Cyclic covers}\label{cyclic_covers}

Let $X$ be a smooth, projective, and connected curve of genus $g$, and let $\sigma:X\to\bP^1$ be a cyclic cover of degree $r$, ramified over $Q_1,\ldots,Q_m\in \bP^1$. Let $\xi=(\xi_1,\ldots,\xi_m)\in (\bZ/r\bZ)^m$ be the tuple of monodromy elements over the $Q_i$, and let $\cH_{g,\bZ/r\bZ,\xi}$ be the moduli space of such covers, which is non-empty and smooth of dimension $m-3$ as long as $\sum_i \xi_i=0$ in $\bZ/r\bZ$ and $\gcd(\{\xi_1, \ldots, \xi_m, r\})=1$, see \cite[Theorem 3.7 and Proposition 3.20]{svz}. We will assume throughout that $m\ge3$, but the results still hold in the case $m=2$, with minor modifications. Taking $\xi_i\in\{0,1,2,\ldots,r-1\}$, we may express $X$ birationally as the curve given by the equation $y^r=(x-\lambda_1)^{\xi_1}\cdots(x-\lambda_{m})^{\xi_m}$, and $\sigma$ as the projection to the $x$-coordinate.

For each $i$, let $f_i=\gcd(\xi_i,r)$, and denote by $P_{i,j}\in X$, for $j=1,2,\ldots,f_i$, the pre-images of $Q_i$ under $\sigma$. Let $e_i=\frac{r}{f_i}$ be the common ramification index of the $P_{i,j}$. By Riemann-Hurwitz, we have
\begin{equation}\label{eq:RH_cyclic}
2g-2=\sum_{i=1}^{m}f_i(e_i-1)-2r=r(m-2)-\sum_{i=1}^{m}f_i.
\end{equation}

We remark that the hyperelliptic locus $\cH_g \subset \cM_g$ is the image of the map $\cH_{g,\bZ/2\bZ,(1^{2g+2})}\to\cM_g$ remembering the source of the cover, see \cite[\S 4.4]{svz}.

\subsection{$\bZ/r\bZ$-equivariant covers} \label{Zr_equivariant_covers}

Let $g:\bP^1\to\bP^1$ be the degree $r$ map totally ramified over $0,\infty$. Consider degree $d$ covers $h:X\to\bP^1$ equivariant under the $\bZ/r\bZ$-actions induced by $\sigma$ and $g$, so that we have a commutative diagram
\begin{equation*}
\xymatrix{
X \ar[r]^{\sigma} \ar[d]^{h} & \bP^1 \ar[d]^{h'}\\
\bP^1 \ar[r]^{g} & \bP^1
}
\end{equation*}
where $h':\bP^1\to\bP^1$ also has degree $d$, and $\sigma:X \to \bP^1$ is a cyclic cover of degree $r$ as defined in the previous section.

For each $i$, the ramification points $P_{i,j}$ of $\sigma$ must all map to either 0 or $\infty$ under $h$; let $S_0$ be the set of $i$ such that $h(P_{i,j})=0$ and similarly define $S_\infty$. For $i\in S_0$, let $a_i>0$ be the common ramification index of the $P_{i,j}$ over 0, and for $i\in S_\infty$, let $-a_i>0$ be the ramification index of $P_{i,j}$ over $\infty$.

Suppose in addition that there are $t_0$ points not among the $P_{i,j}$ mapping to $0$ under $h$ and similarly define $t_\infty$. We assume that all such points are unramified. The sets of pre-images of $0,\infty$ outside of the $P_{i,j}$ are disjoint unions of free $\bZ/r\bZ$-orbits, so in particular $r|t_0,t_\infty$. Note that we must have
\begin{equation}\label{eq:specialfibers_deg}
d=t_0+\sum_{i\in S_0}a_if_i=t_\infty+\sum_{i\in S_\infty}(-a_if_i).
\end{equation}

Suppose now that $h$ has $K=kr$ additional points of ramification, which decompose into $k$ orbits of size $r$, and that no two such points appear in the same fiber of $h$. Let $B=\{b_1,\ldots,b_k\}$ denote the (unordered) multiset of ramification indices (greater than 1) that appear, and let $c_\ell$ be the number of the $b_1,\ldots,b_k$ equal to $\ell+1$. Let $L$ be the set of $\ell$ for which $c_\ell\neq0$. 

Let $\cH_{g,\bZ/r\bZ,\xi,B}$ be the space of $\bZ/r\bZ$-equivariant covers as described above; we suppress the data of the $a_i$ and the values of $t_0,t_\infty$ for notational convenience. We have a quasi-finite and \'{e}tale map $\delta_B:\cH_{g,\bZ/r\bZ,\xi,B}\to\cP_{k,r}$, where $\cP_{k,r}$ is the smooth variety parametrizing unordered collections of $k$ distinct $\bZ/r\bZ$ orbits on $\bP^1$, up to the $\bC^{*}$ action preserving $0,\infty$. In particular, $\cH_{g,\bZ/r\bZ,\xi,B}$ is smooth of dimension $k-1$. We are interested in computing degrees of the forgetful maps $\phi_B:\cH_{g,\bZ/r\bZ,\xi,B}\to \cH_{g,\bZ/r\bZ,\xi}$; in order for $\phi_B$ to expect a non-zero answer, we require
\begin{equation}\label{eq:k_vs_m}
    k=m-2.
\end{equation}

Set $t=t_0+t_\infty$. By Riemann-Hurwitz, we have 
\begin{equation}\label{eq:RH_h}
2g+2d-2=\sum_{i=1}^{m}(|a_i|-1)f_i+\sum_{j=1}^{k}r(b_j-1).
\end{equation}
Substituting $2d=t+\sum_{i=1}^{m}|a_i|f_i$ and $2g-2=r(m-2)-\sum_{i=1}^{m}f_i$, and $k=m-2$, we find that
\begin{equation}\label{eq:t_vs_b}
t=r\sum_{j=1}^{k}(b_j-2).
\end{equation}

\subsection{Congruence condition on $a_i$} \label{congruence_condition}

In this section, we show that in order for equivariant covers $h$ to exist, the vanishing orders $a_i$ need to satisfy a certain congruence condition. We give two perspectives, one algebraic and one topological.

Consider the map $h'$. The fiber over 0 consists of the points $Q_i$ previously defined for $i\in S_0$, with ramification index $a_i$, and $t_0/r$ points of ramification index $r$ each. Similarly, the fiber over $\infty$ consists of $Q_i$ for $i\in S_\infty$ with ramification index $-a_i$ and $t_\infty/r$ points of ramification index $r$ each. Moreover, it is ramified at $k$ additional (unlabelled) points with ramification index $b_j$.

We may therefore represent the map $h'$ by the meromorphic function 
\begin{equation*}
    h'(x)=\prod_{i=1}^{m}(x-\lambda_i)^{a_i}\cdot\left(\frac{P_0(x)}{P_{\infty}(x)}\right)^r
\end{equation*} 
where $P_0, P_\infty$ are non-zero polynomials of degree $t_0/r,t_\infty/r$, respectively. We also allow one of $P_0, P_\infty$ to have strictly smaller degree, which occurs when $h'$ has a pole or zero at $x=\infty$.

\begin{lem}\label{congruence_lemma}
$h'$ lifts to a $\bZ/r\bZ$-equivariant cover $h:X\to\bP^1$ if and only if $a_i\equiv\xi_i\pmod{r}$ for all $i$. If this is the case, the lift is unique, up to composition with the $\bZ/r\bZ$-action on the target.
\end{lem}

\begin{proof}
In order to lift such an $h'$, the meromorphic function $h$ must send a general point $(x,y)$ on $X$ to an $r$-th root of $h'(x)$, that is, such an $r$-th root must exist in the function field $\bC(X)=\bC(x)[\sqrt[r]{\prod_{i=1}^{m}(x-\lambda_i)^{\xi_i}}]$ of $X$. The field extension $\bC(X)/\bC(x)$ is Galois with Galois group $G\cong\bZ/r\bZ$, corresponding to the cyclic cover $\sigma$. Write $f=(x-\lambda_i)^{\xi_i}$ and let $\tau\in G$ be a generator such that
\begin{equation*}
\tau(\sqrt[r]{f})=\omega\cdot \sqrt[r]{f},
\end{equation*}
for some primitive root of unity $\omega\in\bC$, so that
\begin{equation*}
\tau((\sqrt[r]{f})^a)=\omega^a\cdot (\sqrt[r]{f})^a,
\end{equation*}
for $a=0,1,\ldots,r-1$.

Suppose now that $h'(x)$ has an $r$-th root $\alpha\in\bC(X)$, so that $\alpha^r\in \bC(x)$. Then, $\tau(\alpha)^r=\tau(\alpha^r)\in\bC(x)$, so $\tau(\alpha)^r=\alpha^r$. Thus, $\tau(\alpha)=\omega\cdot\alpha$ for some $r$-th root of unity $\omega$. This immediately implies that $\alpha$ is of the form $f_0\cdot(\sqrt[r]{f})^a$ for some $f_0\in\bC(x)$, so $h'=f_0^r\cdot f^a$. In particular, modulo $r$, the $m$-tuple $(a_i)$ of exponents appearing in $h'$ on the factors $(x-\lambda_i)$ must be a multiple of the $m$-tuple $(\xi_i)$ of exponents appearing in $f$.

Now, suppose that $a_i\equiv a\xi_i\pmod{r}$ for all $i$. Then, up to multiplication by an $r$-th root of unity (corresponding to the cyclic action on $\bP^1$), we need $h(x,y)=y^a\cdot\prod_{i=1}^{m}(x-\lambda_i)^{(a_i-a\xi_i)/r}\cdot\left(\frac{P_0(x)}{P_{\infty}(x)}\right)$. However, in order for such an $h$ to be $\bZ/r\bZ$-equivariant, we need $a\equiv1\pmod{r}$.
\end{proof}

We may also see the necessary condition $a_i\equiv\xi_i\pmod{r}$ from the point of view of monodromy as follows. Suppose first that $a_i>0$, and let $\gamma_z$ be a small loop with basepoint $z_0\in\bP^1$ winding  $a_i$ times counterclockwise around 0; the orientation of $\bP^1$ is given by the complex structure. Then, given a basepoint $y_0\in\bP^1$ such that $h'(y_0)=z_0$, $\gamma_z$ lifts under $h'$ uniquely to a loop $\gamma_y$ winding once around $\lambda_i$. Next, given a basepoint $x_0\in X$ such that $\sigma(x_0)=y_0$, $\gamma_y$ lifts uniquely to a path in $X$ travelling from $x_0$ to $\tau^{\xi_i}\cdot x_0$, where $\tau\in \bZ/r\bZ$ is the generator acting on $T_{x_0}X$ by $e^{2\pi i/r}$.

On the other hand, we may also lift $\gamma_z$ to $X$ via the maps $g$ and $h$: letting $y'_0=h(x_0)$, we find that $\gamma_z$ lifts under $g$ uniquely to a path $\gamma'_y$ in $\bP^1$ travelling from $y'_0$ to $\tau^{a_i}\cdot y'_0$. Now, we must have $h(\tau^{\xi_i}\cdot x_0)=\tau^{a_i}\cdot y'_0$, from which we conclude $a_i\equiv \xi_i\pmod{r}$ from the equivariance of $h$.

If $a_i<0$ instead, we may follow the same argument, except now $\gamma_z$ will be a path with winding number $-a_i$ around $\infty$, and the monodromy of $g$ around $\infty$ is the opposite of that around 0.

We have therefore shown that:

\begin{cor}\label{cor:cong}
The space of equivariant covers $\cH_{g,\bZ/r\bZ,\xi,B}$ is empty unless $a_i\equiv\xi_i\pmod{r}$ for all $i$.
\end{cor}

Henceforth we assume this congruence condition on the $a_i$.

\subsection{Reduction to $h'$}\label{reduction_to_genus0}

The main idea which allows us to count the number of $\bZ/r\bZ$-equivariant covers $h:X \to\bP^1$, is that we can count maps $h':\bP^1\to\bP^1$ such that $h' \circ \sigma=g \circ h$ instead. In this section we formulate this statement in terms of moduli spaces.

Let $\cH_{\rat}$ be the moduli space of covers $h':\bP^1\to\bP^1$ with ramification profiles dictated by $h$ upon passing to the scheme-theoretic $\bZ/r\bZ$-quotients. That is, the zeroes and poles of $h'$ are given by the $\lambda_i$ and the roots of $P_0, P_\infty$ and have orders exactly $|a_i|,r$, respectively, and $h'$ is branched over $k$ additional \textit{distinct} points with orders governed by $B$.

The source $\bP^1$ has a fixed choice of coordinate $x$, whereas the target $\bP^1$ has fixed points $0,\infty$, but the functions $h'$ are taken only up to a constant factor. That is, two covers are considered isomorphic if they agree up to scaling on the target. 

Now, we have a quasi-finite and \'{e}tale morphism $\delta_{\rat}:\cH_{\rat}\to\cP_k$, where $\cP_k$ is the space of collections of $k$ unordered branch points on the target, up to scaling. Thus, $\cH_{\rat}$ is smooth of dimension $k-1=m-3$.

On the other hand, we have a morphism $\phi_{\rat}:\cH_{\rat}\to(\bP^1)^{m-3}$ remembering the $\lambda_i$ for $i\ge4$. The source and target have the same dimension, and by generic smoothness, the general fiber of $\phi_{\rat}$ consists of a finite number of reduced points equal to the degree of $\phi_{\rat}$. The above discussion shows that $\deg(\phi_{\rat})=\deg(\phi_B)$.

\subsection{Main result}
We now restate the main theorem.

\begin{thm}\label{main_thm}
Assume the numerical conditions (\ref{eq:RH_cyclic}), (\ref{eq:specialfibers_deg}), (\ref{eq:k_vs_m}), (\ref{eq:RH_h}), (\ref{eq:t_vs_b}), as well as the congruence condition given in Corollary \ref{cor:cong}, so that the map  $\phi_B:\cH_{g,\bZ/r\bZ,\xi,B}\to \cH_{g,\bZ/r\bZ,\xi}$ has non-empty source and target of the same dimension. Then, the degree of $\phi_B$ is
\begin{equation*}
\binom{t/r}{t_0/r}\cdot \prod_{j=1}^{k}(b_j-1)\cdot\frac{k!}{\prod_{\ell\in L}c_\ell!}.
\end{equation*}
\end{thm}

In the case of hyperelliptic odd covers for which all Weierstraß points live over $\infty$ (with fixed ramification indices $a_i$), we have $r=2$, $m=2g+2$, $\xi_i\equiv a_i\equiv1\pmod{2}$ for all $i$, $d=t=t_0=4g$, $k=2g$, and $b_j=3$ for all $j$. Thus, $L=\{2\}$ and $c_{2}=2g$, so we find that $\deg(\phi)=2^{2g}$, recovering the main result of \cite{mp}.

Suppose instead that $r$ is arbitrary and $B=\{2,\ldots,2\}$, that is, the moving ramification conditions are generic (and thus vacuous), we must have $t=0$ by (\ref{eq:t_vs_b}). Then, in the notation of \S\ref{congruence_condition}, we must have simply $h'(x)=\prod_{i=1}^{m}(x-\lambda_i)^{a_i}$, and indeed Theorem \ref{main_thm} gives $\deg(\phi_B)=1$.

\subsection{Theta characteristics}
We comment here on one additional aspect of the case $r=2$ which plays an important role in the computation of \cite{mp}, namely that of the theta characteristic associated with a hyperelliptic odd cover, and how the situation changes in our more general setting.

Recall that if $h:X\to\bP^1$ is any cover for which all ramification indices are odd, one may naturally associate to $h$ a theta characteristic on $X$, that is, a square root of the canonical bundle (see \cite[Appendix B]{acgh} and \cite{mum} for the basic theory of theta characteristics). Namely, we have
\begin{equation*}
    \omega_{X}\cong h^{*}\cO_{\bP^1}(-2)\otimes\cO_X(2R).
\end{equation*}
where $2R$ denotes the ramification divisor of $h$, which by assumption is a square. Thus, the divisor $$\Theta_h:=h^{*}\cO_{\bP^1}(-1)\otimes\cO_X(R)$$ is a theta characteristic on $X$.

Now, suppose $X$ is hyperelliptic and $h$ is equivariant for the hyperelliptic action, that is, $h$ is a point of $\cH_{g,\bZ/2\bZ,(1^{2g+2}),B}$. Assume further that the $|a_i|$ are all odd and that all $b_j$ are equal to 3. This situation generalizes that of \cite{mp} in that Weierstraß points are allowed over 0.

In this case, we can write
\begin{align*}
    h^{-1}(0)&=\sum_{i\in S_0}(2n_i+1)P_i+\frac{t_0}{2}H,\\
    h^{-1}(\infty)&=\sum_{i\in S_\infty}(2n_i+1)P_i+\frac{t_\infty}{2}H
\end{align*}
where $H$ is the hyperelliptic divisor on $X$.

If we denote by $R'$ the ramification divisor outside $h^{-1}(0)$ and $h^{-1}(\infty)$, we have $\deg(R')=4k$, where as before, $k$ is the number of $\bZ/r\bZ$-orbits of branch points of outside $0$ and $\infty$. In particular, $R'=2kH$. Thus, we find, up to linear equivalence,
\begin{align*}
    \Theta_h&=kH + \sum_{i \in S_0,S_\infty} n_i P_i - \sum_{i \in S_0} (2n_i+1) P_i - \frac{t_0}{2} H\\
    &=\left(k-\frac{t_0}{2}\right)H+ \sum_{i \in S_\infty} n_i P_i - \sum_{i \in S_0} (n_i+1) P_i.
\end{align*}

In the covers of \cite{mp}, we have $k=2g$, $t_0=4g$, $t_\infty=0$, and $S_0=\emptyset$, so in fact $\Theta_h$ is \textit{effective}. In our more general situation in which $S_0$ may be non-empty, note that even if we retain the conditions $(t_0,t_\infty)=(4g,0)$ and $k=2g$, any (possibly non-effective) theta characteristic of $X$ arises as above. Indeed, if
\begin{equation*}
    \Theta=\sum_{i=1}^{2g+2}c_iP_i,
\end{equation*}
then we may take $S_\infty$ to be the set of $i$ such that $c_i\ge0$, and $n_i=c_i$ for $i\in S_\infty$ and $n_{i}=-c_i-1$ for $i\in S_0$.

\section{Counting meromorphic functions on $\bP^1$} \label{counting_meromorphic_functions}

In this section, we carry out the proof of Theorem \ref{main_thm}. We have reduced to counting the possible meromorphic functions 
\begin{equation*}
h'(x)=\prod_{i=1}^{m}(x-\lambda_i)^{a_i}\cdot\left(\frac{P_0(x)}{P_{\infty}(x)}\right)^r,
\end{equation*}
where we impose the condition that, away from the $\lambda_i$ and the roots of $P_0, P_\infty$, the cover $h'$ has ramification indices (at distinct branch points) given by the multi-set $B$. We parametrize choices of $P_0, P_\infty$ by a product of projective spaces, and impose the ramification conditions from $B$ in terms of a combination of Segre and Veronese embeddings.

\subsection{The map $\psi$}

Let $\bP^{t_0/r},\bP^{t_\infty/r}$ be the projective spaces parametrizing non-zero polynomials $P_0, P_\infty$ of degrees (at most) $t_0/r,t_\infty/r$, up to scaling, and let $U\subset\bP^{t_0/r}\times\bP^{t_\infty/r}$ be the open locus consisting of polynomials $P_0, P_\infty$ satisfying the following properties:
\begin{itemize}
\item $P_0, P_\infty$ have no roots in common,
\item $P_0, P_\infty$ do not have roots at any of the $\lambda_i$,
\item $P_0, P_\infty$ have no double roots.
\end{itemize}
Here, we regard $x=\infty$ as a possible root of $P_0, P_\infty$. Thus, the statement that $P_0, P_\infty$ have no roots in common includes the statement that at most one of them has zero leading coefficient, and the statement that they have no double roots includes the statement that the first two coefficients of either polynomial do not both vanish.

If $(P_0, P_\infty)\in U$, the ramification divisor for $h'$ away from the $\lambda_i$ and the roots of $P_0, P_\infty$ is given by the roots of the polynomial
\begin{align*}
\psi_{h'}(x):&=\frac{1}{\prod_{i=1}^{m}(x-\lambda_i)^{a_i-1}}\cdot\frac{P_\infty^{r+1}}{P_0^{r-1}}\cdot\frac{d}{dx}h'(x)\\
&=\sum_{i=1}^{m} a_i\cdot\frac{\prod_{i=1}^{m}(x-\lambda_i)}{x-\lambda_i}\cdot P_0\cdot P_\infty+r\prod_{i=1}^{m}(x-\lambda_i)\cdot(P'_0P_\infty-P_0P'_\infty)
\end{align*}

The degree of $\psi_{h'}(x)$ is $b=\sum_j(b_j-1)$ (note that the leading terms of the two summands cancel), unless $\infty$ is among the additional ramification points of $h'$, in which the ramification index at $\infty$ is $b-\deg(\psi_{h'})+1$. 

Note that the definition of $\psi_{h'}(x)$ makes sense even for $(P_0, P_\infty)\notin U$, and that the roots of $\psi_{h'}(x)$ away from $\lambda_i$ and the roots of $P_0, P_\infty$, taken with multiplicity, still control the ramification of $h'$. In this case, $\psi_{h'}(x)$ will in general have extraneous roots among $\lambda_i$ and the roots of $P_0, P_\infty$.

We have always that $\psi_{h'}(x)$ is not identically zero, or else $h'$ is constant, which is impossible even for $(P_0, P_\infty)\in U$ because the $a_i$ are not divisible by $r$. Therefore, we may make the following definition.

\begin{defn}
We define $\psi:\bP^{t_0/r}\times\bP^{t_\infty/r}\to\bP^b$ by $\psi((P_0, P_\infty))=\psi_{h'}(x)$.
\end{defn}

The map $\psi$ is the composition of a Segre embedding followed by a linear projection whose base locus does not meet the image of $\bP^{t_0/r}\times\bP^{t_\infty/r}$, so its degree is equal to that of the Segre embedding, which is $\binom{t/r}{t_0/r}$.

\subsection{The map $\rho$}

We want to intersect $\psi$ with the locus of polynomials that arise as a product of the form $\prod_{\ell\in L}R_{\ell}(x)^\ell$, where $R_\ell$ has degree $c_\ell$. This is the image of the map defined below:
\begin{defn}
We define the map
\begin{equation*}
\rho:\prod_{\ell\in L}\bP^{c_\ell}\to\bP^b
\end{equation*}
by sending $(R_{\ell})_{\ell\in L}\mapsto \prod_{\ell\in L} R_\ell^{c_\ell}$. 
\end{defn}

The map $\rho$ is the composition of the $\ell$-Veronese embedding on the factor corresponding to $\ell\in L$, followed by a Segre embedding, followed by a linear projection. Because $\rho$ has no base locus, its degree is the degree of the composition of the first two maps, which is
\begin{equation*}
\prod_{\ell\in L}\ell^{c_{\ell}}\cdot\frac{(\sum_{\ell\in L}c_\ell)!}{\prod_{\ell\in L}c_\ell!}=\prod_{j=1}^{k}(b_j-1)\cdot\frac{k!}{\prod_{\ell\in L}c_\ell!}.
\end{equation*}

\subsection{Intersection of $\psi$ and $\rho$}

Theorem \ref{main_thm} now follows from B\'{e}zout's Theorem, assuming that, for $\lambda_i$ general, the intersections of $\psi$ and $\rho$ correspond exactly to points in the corresponding fiber of $\cH_{rat}$, and that this intersection is transverse. These statements are verified below, completing the proof of the main result.

\begin{prop}
Suppose the $\lambda_i\in\bC$ are general and consider the intersection of  $\psi,\rho$ in $\bP^b$. Then, any cover $h':\bP^1\to\bP^1$ defined by a point of intersection between $\psi,\rho$ is a point of $\cH_{\rat}$.
\end{prop}

\begin{proof}
Suppose we have $h'(x)=\prod_{i=1}^{m}(x-\lambda_i)^{a_i}\cdot\left(\frac{P_0(x)}{P_{\infty}(x)}\right)^r$ defining a point of intersection of $\psi,\rho$ outside of  $\cH_{\rat}$. We claim that the corresponding map $h':\bP^1\to\bP^1$, possibly of degree less than $d$, has strictly fewer than $k$ branch points away from $0,\infty$, and thus varies in a family of dimension strictly less than that of $\cH_{\rat}$, so is avoided by a general choice of $\lambda_i$. (If $k=1$, then we find that $h'$ may only have two branch points, and hence two ramification points, but we see that this is impossible when $m\ge3$, as $r\neq a_i$.)

To see this, note first that if $h'$ has $k$ distinct ramification points away from $0,\infty$, then the ramification indices must be given exactly by the elements of $B$. Now, we claim that any choices of $P_0, P_\infty$ that define a cover $h'$ outside of $\cH_{\rat}$ increases the total amount of ramification of $h'$ over 0 and $\infty$ \textit{compared to the change in degree}, which cannot happen if the ramification away from $0,\infty$ is unchanged.

On one hand, suppose that the degree of $h'$ is $d$, so that $P_0, P_\infty$ have no roots in common and no roots at the $\lambda_i$. Then, the only way $h'$ can fail to define a point of $\cH_{\rat}$ is if some of the ramification imposed by the $\lambda_i$ and $P_0, P_\infty$ coalesces, either by allowing $P_0, P_{\infty}$ to have multiple roots or roots at $\lambda_i$. However, this increases the total amount of ramification of $h'$.

On the other hand, if $P_0, P_\infty$ are chosen to decrease the degree of $h'$, any such choice will decrease the degree of $h'$ strictly more than the amount of ramification lost. For example, if $P_0, P_\infty$ have a linear factor in common, the degree of $h'$ drops by $r$, which must decrease the total ramification by $2r$, but the amount of ramification lost over $0,\infty$ is only $2r-2$. Similarly, if $P_0$ or $P_\infty$ has a root at $\lambda$ canceling a factor of $(x-\lambda)$, we get a similar contradiction.

Therefore, if the $\lambda_i$ are chosen generally, the ramification for all covers in the intersection of $\psi,\rho$ is as expected.
\end{proof}

\begin{prop}
Suppose the $\lambda_i\in\bC$ are general. Then, the intersection of $\psi,\rho$ is transverse.
\end{prop}

\begin{proof}
Suppose $\psi,\rho$ have a tangent vector. We then obtain a 1-parameter deformation 
\begin{equation*}
    \wt{h'}=\prod_{i=1}^{m}(x-\lambda_i)^{a_i}\cdot\left(\frac{P_0+\epsilon Q_0}{P_\infty+\epsilon Q_{\infty}}\right)^{r}
\end{equation*}
of the meromorphic function $h'$, which yields a 1-parameter deformation as a point of $\cH_{\rat}$. By construction, $\wt{h'}$ keeps the $\lambda_i$ constant on the target. Because the $\lambda_i$ are general and $\phi_{\rat}$ is generically unramified, $\wt{h'}$ must is in fact be the trivial deformation on $\cH_{\rat}$.

This means that, as meromorphic functions on $\bP^1\times\Spec\bC[\epsilon]/\epsilon^2$, we have 
\begin{equation*}
    \prod_{i=1}^{m}(x-\lambda_i)^{a_i}\cdot\left(\frac{P_0+\epsilon Q_0}{P_\infty+\epsilon Q_{\infty}}\right)^{r}=\mu\prod_{i=1}^{m}(x-\lambda_i)^{a_i}\cdot\left(\frac{P_0}{P_\infty}\right)^{r}
\end{equation*}
for a scalar $\mu\in \bC[\epsilon]/\epsilon^2$. We find that $Q_0,Q_\infty$ must be scalar multiples of $P_0, P_\infty$, corresponding to the trivial tangent vector on $\psi$, and the ramification remains unchanged to first order, corresponding to the trivial tangent vector on $\rho$. This completes the proof.
\end{proof}


\begin{thebibliography}{AMSa}
\bibitem[ACGH85]{acgh} Enrico Arbarello, Maurizio Cornalba, Phillip A. Griffiths and Joseph D. Harris, \emph{Geometry of algebraic curves. Vol. I.}, Springer-Verlag, New York (1985).

\bibitem[FP05]{fp} Carel Faber and Rahul Pandharipande, \emph{Relative maps and tautological classes}, J. Eur. Math. Soc. \textbf{7} (2005), 13-49.

\bibitem[FMNP20]{fmnp} Gavril Farkas, Riccardo Moschetti, Carlos Naranjo, Gian Pietro Pirola, \emph{Alternating Catalan numbers and curves with triple ramification}, Ann. Sc. Norm. Super. Pisa, to appear.

\bibitem[HM82]{hm} Joe Harris, David Mumford, \emph{On the Kodaira dimension of the moduli space of curves}. Invent. Math. \textbf{67} (1982), 23-86.

\bibitem[L19]{lian_pencils} Carl Lian, \emph{Enumerating pencils with moving ramification on curves}, J. Algebraic Geom., to appear.

\bibitem[MP20]{mp} Riccardo Moschetti and Gian Pietro Pirola, \emph{Hyperelliptic odd coverings}, Israel J. Math., to appear.

\bibitem[Mum71]{mum} David Mumford, \emph{Theta characteristics of an algebraic curve}, Ann. Sci. \'{E}c. Norm. Sup\'{e}r. \textbf{4} (1971), 181-192.

\bibitem[SvZ18]{svz} Johannes Schmitt and Jason van Zelm, \emph{Intersection of loci of admissible covers with tautological classes}, Selecta Math. (N.S.) \textbf{26}, Article Nr. 79 (2020).

\end{thebibliography}
\end{document}